\documentclass[11pt,b5paper,twoside, headrule]{amsart}

\usepackage{amsfonts, amsmath, amssymb,latexsym}
\usepackage{epsfig}
\usepackage[curve]{xy}
\usepackage{algorithmic}
\usepackage{algorithm}
\usepackage{enumerate}
\usepackage{framed}
\usepackage{hyperref}
\usepackage{mathtools}

\usepackage{chngcntr}

\footskip=10mm\parskip 0.2cm\addtocounter{page}{0}
\newtheorem{thm}{Theorem}[section]
\newtheorem{cor}[thm]{Corollary}
\newtheorem{lem}[thm]{Lemma}
\newtheorem{prop}[thm]{Proposition}

\theoremstyle{definition}
\newtheorem{defn}[thm]{Definition}

\newtheorem{example}[thm]{Example}

\newtheorem{prob}[thm]{Problem}

\theoremstyle{remark}
\newtheorem{rem}[thm]{Remark}
\numberwithin{equation}{section}

\title[On the RLWE/PLWE equivalence for cyclotomic number fields]{On the RLWE/PLWE equivalence for cyclotomic number fields}

\author{\sc Iv\'an Blanco-Chac\'on}
\address{Department of Mathematics, School of Science\\
Universidad de Alcal\'a de Henares\\
Ctra. Madrid-Barcelona Km. 33,600\\
Alcal\'a de Henares, Spain}
\email{ivan.blancoc@uah.es}
\thanks{Partially supported by MTM2016-79400-P}


\begin{document}
\renewcommand\baselinestretch{1.2}
\renewcommand{\arraystretch}{1}
\def\base{\baselineskip}
\font\tenhtxt=eufm10 scaled \magstep0 \font\tenBbb=msbm10 scaled
\magstep0 \font\tenrm=cmr10 scaled \magstep0 \font\tenbf=cmb10
scaled \magstep0


\def\evenhead{{\protect\centerline{\textsl{\large{I. Blanco}}}\hfill}}

\def\oddhead{{\protect\centerline{\textsl{\large{On the non vanishing of the cyclotomic $p$-adic $L$-functions}}}\hfill}}

\pagestyle{myheadings} \markboth{\evenhead}{\oddhead}

\thispagestyle{empty}

\maketitle

\begin{abstract}We study the equivalence between the Ring Learning With Errors and Polynomial Learning With Errors problems for cyclotomic number fields, namely: we prove that both problems are equivalent via a polynomial noise increase as long as the number of distinct primes dividing the conductor is kept constant. We refine our bound in the case where the conductor is divisible by at most three primes and we give an asymptotic subexponential formula for the condition number of the attached Vandermonde matrix valid for arbitrary degree.
\end{abstract}

\bigskip
\section{Introduction}
The Ring Learning With Errors problem in its several fashions (primal, dual and polynomial, decisional or search) constitutes the basis of some of the most promising and versatile public key cryptosystems, signature and key exchange protocols for the postquantum era. This is easily seen in the list of the surviving contenders (at the time of writing) in the last NIST public contest, which started in November 2017: in January 2019, the results of its second round were made public and taking into account the attacks and feedback to the surviving proposals of the first round, 26 proposals  passed this new sieve. The numbers of remaining proposals  within each category are listed in Table \ref{tabla:sencilla}.

\begin{table}[htbp]
\begin{center}
\begin{tabular}{|l|l|}
\hline
Category & Number of proposals\\
\hline \hline
Code-based (Hamming) &  5 \\ \hline
Code-based (rank metric) &  2 \\ \hline
Lattice-based (LWE)  &  1 \\ \hline
Lattice-based (RLWE)  &  6 \\ \hline
Lattice-based (PLWE)  &  1 \\ \hline
Lattice-based (Other)  &  4 \\ \hline
Multivariate-based & 4 \\ \hline
Hash-based &  1 \\ \hline
Supersingular isogeny-based &  1 \\ \hline
Other &  1 \\ \hline
\end{tabular}
\caption{NIST proposals. Second Round.}
\label{tabla:sencilla}
\end{center}
\end{table}

As we can see, the lattice-based category, and within it, the RLWE/PLWE subcategory remains the strongest contender in terms of number of surviving proposals.\footnote{At  https://www.safecrypto.eu/pqclounge/ a summary of candidates and the history of all submissions, attacks and withdrawals is available to filter and check.} These numbers justify the enormous interest in the investigation of the various open ends in RLWE. One of these open ends is the relation between two of its versions: the RLWE problem, which is formulated in terms of rings of integers of algebraic number fields and the PLWE version, in terms of rings of polynomials. Even when the number field is monogenic, namely, the rings defining the two problems are isomorphic, the error distributions can be very different in both scenarios. Although equivalences between dual/primal RLWE, dual/primal PLWE, decisional/search RLWE and decisional/search PLWE have been established, the equivalence between RLWE and PLWE has only been shown for a restricted although infinite class of number fields (see \cite{RSW} and \cite{DD}). Even in the cyclotomic case, which backs the seminal papers \cite{stehle2} and \cite{LPR}, very little is known -or formally proved and published- apart from the power-of-two case, used in \cite{stehle2}, and for which the distortion between the canonical and the coordinate embedding is a scaled isometry, as we recall in Section 3.

This problem, even for the cyclotomic cases seems a difficult one. The ideas in \cite{DD} can be applied to show the equivalence for cyclotomic number fields of degree $2^kp$ or $2^kpq$ with $p,q$ primes and $q<p$. But apart from these cases and the ad-hoc family constructed in \cite{RSW} nothing else is known on PLWE/RLWE equivalence at the time of writing. And there is a good reason to be interested in such an equivalence: in \cite{bernstein}, it is shown how the arithmetic of several polynomials rings leads to very efficient cryptographic designs, making the polynomial rings more amenable for computer implementations than ideals in rings of integers of number fields.

Our contribution is as follows: first, we prove that RLWE and PLWE are equivalent via a polynomial noise increase as long as the number of distinct primes dividing the conductor is constant. Moreover, by using results due to Bang (\cite{bang}) and Bloom (\cite{bloom}), we can give a sharper upper bound in the cases where the conductor is the product of at most four primes and show that the general power case can be reduced to the square-free case in a surprising way: up to a square factor in the degree, only the radical contributes to the noise increase. Secondly, we give a proof of a general asymptotic subexponential equivalence. Our methods are different from those in \cite{DD}, since we purely use several known algebraic properties of the cyclotomic polynomials, some arithmetic estimates for the divisor function and a careful analysis of the coefficients of the cyclotomic polynomials in terms of the roots via Vieta's formulas, an idea which was used in \cite{RSW}, but the precise shape of the polynomials they start with makes the use of Vieta's formulas more or less direct, something which does not happen in the cyclotomic case.

The organisation of our work is as follows: in Section 2 we give a brief summary of the key concepts of algebraic number theory necessary to introduce the RLWE/PLWE, what we do afterwards. We focus on the study of cyclotomic number fields for obvious reasons. We also recall what is understood by \emph{equivalence}, and how it relates to the condition number.  In Section 3 we start by recalling the equivalence in the power of two cyclotomic case (proof included for the convenience of the reader) and for the family studied in \cite{RSW}. After that, in Theorem 3.10 we give a general asymptotic bound for the condition number in terms of the maximum of the coefficients of the cyclotomic polynomial in absolute value and the degree of the cyclotomic field. This bound is polynomial in the degree when the number of prime factors is constant (Corollary 3.11). Section 4 refines the former results when the number of different prime factors of the degree is at most $3$ and Section $5$ closes our study by proving that in general, the growth of the condition number is at most subexponential (Theorem 5.4) in the degree, thanks to an asymptotic upper bound for the maximal coefficient, due to Bateman, and some asymptotic estimates for the prime divisor function $\omega$.

Finally, the author must express his gratitude to the anonymous referee for the careful review of our preliminary version. His corrections and suggestions have significantly strengthened our main results and helped re-organise the manuscript into a version which is more amenable and comfortable to read

\section{Prolegomena}

We start with the definition of lattice that we will use in our work:

\begin{defn}A lattice in $\mathbb{R}^n$ is a subgroup $\Lambda$ of the additive group $(\mathbb{R}^n,+)$ which is isomorphic to $\mathbb{Z}^n$ and such that $\Lambda\otimes_{\mathbb{Z}}\mathbb{R}=\mathbb{R}^n$.
\end{defn}
This definition has implicit the feature of being of full rank and hence, all lattices will be full rank for us. There are more general definitions but this will be enough for us. 
\begin{prob}[The approximate Shortest Vector Problem] For $\gamma>0$, the $\gamma$-approximate shortest vector problem ($\gamma$-SVP) is, on input of a full rank arbitrary lattice $\Lambda$ together with a $\mathbb{Z}$-basis, to determine a non-zero vector $x\in\Lambda$ with length smaller than $\gamma\lambda_1(\Lambda)$.
\end{prob}

It is proved in \cite{boas} that $\gamma$-SVP is NP-hard for small enough $\gamma$. In \cite{regev}, Regev introduced the Learning With Errors Problem (LWE) and established a quantum polynomial time reduction from $\gamma$-SVP to LWE but for much larger approximation factors than in \cite{boas}. Moreover Regev's cryptosystem presents a quadratic overhead in the size of the public key which renders it unfeasible in scenarios where high speed computations are required over a very large plaintext set, such as election systems (e-voting/i-voting). To tackle this unfeasibility, Stehl\'e et al (\cite{stehle2}) and Lyubashevsky et al (\cite{LPR}) introduced the PLWE and RLWE problems, whose security is backed on a restricted version of SVP, as we will recall in Section 2.2.

\subsection{Relevant facts of algebraic number theory}

Readers who are familiar with this material can safely skip it as all our notations are standard. Only definitions and facts which are essential for our proofs are recalled here; in particular we have omitted the definitions of trace, norms and discriminants. Readers who are not so familiar with this material  are referred to \cite{stewart}, Chapter 2 or to any standard first year textbook in algebraic number theory. 

\subsubsection{Algebraic number fields} An algebraic number field is a field extension $K = \mathbb{Q}(\theta)/\mathbb{Q}$ of some finite degree $n$, where $\theta$ satisfies a relation $f(\theta) = 0$ for some irreducible polynomial $f(x) \in\mathbb{Q}[x]$, which is monic without loss of generality. The polynomial $f$ is called the minimal polynomial of $\theta$, and $n$ is also the degree of $f$. Notice that $K$ is in particular an $n$-dimensional $\mathbb{Q}$-vector space and the set $\{1,\theta,...,\theta^{n-1}\}$ is a $\mathbb{Q}$-basis of $K$ called a power basis. Notice that associating $\theta$ with the indeterminate $x$ yields a natural isomorphism between $K$ and $\mathbb{Q}[x]/f(x)$.

A number field $K = Q(\theta)$ of degree $n$ has exactly $n$ field embeddings (field monomorphisms) fixing the base field $\mathbb{Q}$, which we denote $\sigma_i: K \to \overline{\mathbb{Q}}$, where $\overline{\mathbb{Q}}$ stands for an algebraic closure of $\mathbb{Q}$, fixed from now on. These embeddings map $\theta$ to each of the roots of its minimal polynomial $f$. The number field is said to be Galois if $K$ is the splitting field of $f$, or equivalently if the images of all the embeddings coincide.

An embedding whose image lies in $\mathbb{R}$ is called a real embedding; otherwise it is called a complex embedding. Since non-real roots of $f$ come in conjugate pairs, so do the complex embeddings. The number of real embeddings is denoted $s_1$ and the number of pairs of complex embeddings is denoted $s_2$, so we have $n=s_1+2s_2$. If $s_2=0$ ($s_1=0$) then $K$ is said to be totally real (totally imaginary).

The canonical embedding $\sigma: K\to \mathbb{R}^{s_1}\times\mathbb{C}^{2s_2}$ is defined as 
$$
\sigma(x) = (\sigma_1(x),...,\sigma_n(x)).
$$ 
\begin{defn}
An algebraic integer is an element of $\overline{\mathbb{Q}}$ whose minimal polynomial over $\mathbb{Q}$ has integer coefficients. 
\end{defn}

Let $\mathcal{O}_K\subset K$ denote the set of all algebraic integers in $K$. This set forms a ring under addition and multiplication in $K$ (\cite{stewart}, Theorem 2.9), called the ring of integers of $K$. It happens that $\mathcal{O}_K$ is a free $\mathbb{Z}$-module of rank $n$, i.e., it is the set of all $\mathbb{Z}$-linear combinations of some (non-unique) basis $\mathcal{B}=\{b_1,...,b_n\}\subset \mathcal{O}_K$ of $K$ (\cite{stewart}, Theorem 2.16). Such a set $\mathcal{B}$ is called an integral basis.

\subsubsection{Ideal lattices} By a discrete ring we mean a ring which is free of finite rank as abelian group. An ideal lattice in $\mathbb{R}^n$ is a lattice of the form $\sigma(I)$ where $I$ is an ideal in a discrete ring $R$ (of rank $n$) and $\sigma: R\cong \mathbb{Z}^n$ is a group isomorphism. By construction, apart from the additive structure, an ideal lattice $\sigma(I)$ has also a product, inherited from the product in $R$. 

When $R=\mathcal{O}_K$ for a number field $K$, the canonical embedding $\sigma$ provides in a natural way an ideal lattice for each ideal $I$ of $R$. For the canonical embedding, multiplication and addition are preserved component-wise. This is not true for the coordinate embedding: for instance, for the ring $\mathbb{Z}[x]/(x^m+1)$, for $m=2^l$, multiplying by $x$  is equivalent to shifting the coordinates and negate the independent term. This is one of several advantages of using the canonical embedding.

\subsubsection{Cyclotomic fields.} Let $n>1$ be an integer. The set of primitive $n$-th roots of unity (those of the form $\theta_k=exp(2\pi i)k/n$, with $1\leq k \leq n$ coprime to $n$) forms a multiplicative group of order $m=\phi(n)$. The $n$-th cyclotomic polynomial is
$$
\Phi_n(x)=\prod_{k\in\mathbb{Z}^*_n}(x-\theta_k).
$$
This is the minimal polynomial of $\theta_k$ for each $k$, so that $K=\mathbb{Q}(\theta_k)$ is an algebraic number field of degree $m$. It can be proved (\cite{stewart} Chap 3) that $\mathcal{O}_K=\mathbb{Z}[\theta]$ with $\theta=\theta_k$ for each $k$.

\begin{prop}Let $p$ be any prime. Then:
\begin{itemize}
\item[a)] If $p=2$, then for any $k\geq 2$, it holds that
$$
\Phi_{2^k}(x)=x^{2^{k-1}}+1.
$$
\item[b)] If $p>2$, then for any $k\geq 1$, it holds that
$$
\Phi_{p^k}(x)=\sum_{i=0}^{p-1}x^{ip^{k-1}}.
$$
\end{itemize}
\label{polycyclo}
\end{prop}
\proof
For $p=2$, the polynomial $x^{2^{k-1}}+1$ is irreducible via Eisenstein criterion after change of variable $x=y+1$. Likewise, this polynomial vanishes at each $2^k$-th primitive root of unity. A similar argument holds for $p>2$ and $k=1$, and using properties of the geometric series, we can reduce to this argument for each $k\geq 1$. For details, check \cite{washington} Ch. II.\qedhere

The following result will also be useful later on:

\begin{prop}[\cite{washington} Ch. II] Let $m=pr$ with $p$ prime and $r$ not divisible by $p$. Then 
$$
\Phi_m(x)=\Phi_r(x^p)/\Phi_r(x).
$$
In addition, if we write $m=p_1^{r_1}\cdots p_l^{r_l}$ with $p_1,\cdots,p_l$ different primes and denote $rad(m):=p_1\cdots p_l$, then
$$
\Phi_m(x)=\Phi_{rad(m)}(x^{\frac{m}{rad(m)}}).
$$
\label{polycyclo2}
\end{prop}

\begin{defn}A number field $K$ such that, like in the cyclotomic case, $\mathcal{O}_K=\mathbb{Z}[\alpha]$ for some $\alpha\in K$ is said to be monogenic.
\label{defmono}
\end{defn}

\subsection{Ring/Polynomial Learning With Errors} Let $K=\mathbb{Q}(\alpha)$ be a number field of degree $n$ and let $\mathcal{O}_K$ be its ring of integers.

\subsubsection{Statement of the problems}

Assume that $K$ is the splitting field of a monic irreducible polynomial $f(x)\in\mathbb{Z}[x]$ with $f(\alpha)=0$ and consider the ring $\mathcal{O}=\mathbb{Z}[x]/(f(x))$. The ring $\mathbb{Z}[\alpha]\cong\mathcal{O}$ has finite index in $\mathcal{O}_K$, and the restriction of the canonical embedding to $\mathcal{O}$ provides a lattice. A very common choice is $f(x)=\Phi_{p^k}(x)$, and even more common is the choice $p=2$ (see \cite{stehle2}).

\begin{defn}[The RLWE/PLWE problem] Let $\chi$ be a discrete random variable with values in $\mathcal{O}_K/q\mathcal{O}_K$ (resp. in $\mathcal{O}/q\mathcal{O}$). The RLWE (resp. PLWE) problem for $\chi$ is defined as follows:

For an element $s\in \mathcal{O}_K/q\mathcal{O}_K$ (resp. $\mathcal{O}/q\mathcal{O}$) chosen uniformly at random, if an adversary is given access to arbitrarily many samples $\{(a_i,a_is+e_i)\}_{i\geq 1}$ of the RLWE (resp. PLWE) distribution, where for each $i\geq 1$, $a_i$ is uniformly chosen at random and $e_i$ is sampled from $\chi$, the adversary must recover $s$ with non-negligible advantage.\footnote{This is the  definition of RLWE/PLWE in \emph{search version}. As all this material is nowadays well known to the specialist we are sparing as many details as possible. We are taking this version as starting point, as it is more suitable for our argument. We refer the reader to \cite{LPR} for the \emph{decisional version} of the problem.}
\end{defn}

In \cite{stehle2}, a polynomial time reduction is given from worst case SVP over ideal lattices to the PLWE problem for power-of-two cyclotomic fields. Later, in \cite{LPR} ideal-SVP polynomial time reduction was established for RLWE over cyclotomic fields under flexible conditions on the security parameters and further, in \cite{PRS} the polynomial reduction was extended to non-cyclotomic Galois number fields building on the same number-theoretical kind of arguments as in \cite{LPR}.

\subsubsection{Equivalence between the formulations.}For a number field $K$, we say that $RLWE$ and $PLWE$ are equivalent for $K$ if every solution for the first can be turned in polynomial time into a solution for the second (and viceversa), incurring in a noise increase which is polynomial in the number field degree. In \cite{RSW}, the equivalence between RLWE and PLWE is proved for the following family of polynomials:

\begin{thm}[\cite{RSW}, pag. 4 and Theorem 4.7] There is a polynomial time reduction algorithm from RLWE over $K_{f_{n,p}}$ to PLWE for $f_{n,p}(x)$ where $K_{f_{n,p}}$ is the splitting field of $f_{n,p}(x)=x^n+xp(x)-r$ where $n\geq 1$, $p(x)$ runs over polynomials with $deg(p(x))<n/2$ and $r$ runs over primes such that $25||p||_1^2\leq r\leq s(n)$, with $s(x)$ a polynomial. Notice that there is a trivial reduction from PLWE to RLWE.\footnote{For $p(x)=\displaystyle\sum_{i=0}^np_ix^i\in\mathbb{R}[x]$, the $1$-norm is defined as $||p||_1=\displaystyle\sum_{i=0}^n|p_i|$}
\end{thm}

The argument to prove this theorem is, first, to consider the family of polynomials $\phi_{n,a}(x):=x^n-a$, with $a\in\mathbb{Z}\setminus\{0\}$ square-free. Denoting by $K_{\phi_{n,a}}$ the splitting field of $\phi_{n,a}(x):=x^n-a$, the authors check in first place the equivalence for $K_{\phi_{n,a}}$ and they show, via a careful use of Rouch\'e theorem, that when $\phi_{n,a}(x)$ is perturbed by adding another polynomial with degree smaller than $n/2$ the roots of both polynomials are close enough.  

\subsubsection{Distortion between embeddings: the condition number.} For a monic irreducible polynomial $f(x)\in \mathbb{Z}[x]$ and $\theta$ a root of $f(x)$, consider the subring $\mathbb{Z}[x]/(f(x))\cong\mathbb{Z}[\theta]\subseteq\mathcal{O}_K$. As lattices, $\mathbb{Z}[x]/(f(x))$ is endowed with the coordinate embedding while $\mathbb{Z}[\theta]$ is endowed with the canonical embedding inherited from $\mathcal{O}_K$, and the evaluation-at-$\theta$ morphism causes a distortion between both. Explicitly, the transformation between the embeddings is given by
\begin{equation}
\begin{array}{ccc}
V_f: \mathbb{Z}[x]/(f(x)) & \to & \sigma_1(\mathcal{O}_{K})\times\cdots\times\sigma_n(\mathcal{O}_{K})\\
\displaystyle\sum_{i=0}^{n-1}a_i\overline{x}^i & \mapsto & \left(\begin{array}{cccc}1 & \theta_1 & \cdots & \theta_1^{n-1}\\ 
1 & \theta_2 & \cdots & \theta_2^{n-1}\\  
\vdots & \vdots & \ddots \vdots\\ 
1 & \theta_n & \cdots & \theta_n^{n-1}\end{array}\right)\left(\begin{array}{c}a_0 \\ a_1 \\  \vdots \\a_{n-1}\end{array}\right),
\end{array}
\label{latticebij}
\end{equation}
where $\overline{x}$ is the class of $x$ modulo $f(x)$ and $\theta=\theta_1,\theta_2,...,\theta_n$ are the Galois conjugates of $\theta$. Namely, the transformation $V_f$ is given by a Vandermonde matrix (which we denote also by $V_f$) acting on the coordinates. 

For any matrix $A=(a_{ij})\in M_{n\times n}(\mathbb{C})$, denote by $A^*$ its transposed conjugate and recall that its Frobenius norm is defined as 
\begin{equation}
||A||:=\sqrt{Tr(AA^*)}=\sqrt{\sum_{i,j=1}^n|a_{ij}|^2}.
\label{easy}
\end{equation}

The noise growth caused by $V_f$ will remain \emph{controlled} whenever $||V_f||$ and $||V_f^{-1}||$ remain so, and as justified in \cite{RSW}, a reasonable measure of how both quantities are controlled is given by $||V_f||||V_f^{-1}||$.
\begin{defn}The condition number of an invertible matrix $A\in\mathrm{M}_n(\mathbb{C})$ is defined as Cond$(A):=||A|||A^{-1}||$.
\end{defn}

Hence, in the monogenic case, the problem of the equivalence is reduced to show that $Cond(V_f)=O(n^r)$ for some $r$ independent of $n$. The non-monogenic case needs an intermediate reduction that we do not discuss here as our results deal with cyclotomic number fields. But even in the monogenic case the difficulty in our problem is that Vandermonde matrices are in general very ill conditioned. 

Indeed, for a sequence of real nodes $s=\{1, s_1,...,s_{n-1}\}$, the corresponding Vandermonde matrix is exponentially conditioned at least in the following cases (\cite{gautschi}):
\begin{itemize}
\item When all the nodes are positive. In this case $Cond(V_s)>2^{n-1}$.
\item When the nodes are symmetrically located with respect to the origin. In this case $Cond(V_s)>2^{n/2}$.
\item For some special but large families: harmonic nodes, rational families of nodes in $[0,1]$ and $[-1,1]$, roots of Chebyshev polynomials and other orthogonal families...
\end{itemize}

The situation in the complex case, however, is very different. For instance, for a sequence of complex nodes $s=\{1, s_1,...,s_{n-1}\}$ , the corresponding Vandermonde matrix tends to be badly conditioned unless the nodes are more or less equally spaced on or about the unit circle (\cite{pan}).

\begin{example}For $n>1$, let $s=\{1,\theta_n,\theta_n^2,...,\theta_n^{n-1}\}$ be the set of all the $n$-th roots of $1$, not just the primitive ones. Then $Cond(V_s)=n$: indeed, according to Eq. \ref{easy}, $||V_s||=n$. On the other hand, it is easy to check that
$$
V_sV_s^*=nId,
$$
hence $V_s^{-1}=n^{-1}V_s^*$, and since $||V_s^*||=||V_s||=n$, then $||V_s^{-1}||=1$.
\label{excyclotot}
\end{example}

So, intuitively, and according with this philosophy, the difficulty of bounding the condition number for cyclotomic number fields, is that even if Vandermonde matrices of full systems of $n$-th roots are linearly conditioned, Vandermonde matrices attached to cyclotomic polynomials contain only the primitive roots, which are not equally spaced on the unit circle. In particular, when the modulus is the product of a large number of prime factors, the geometric distribution of the roots in the unit circle can vary in a very chaotic manner.

\begin{example}Another piece of bad news is that polynomial condition numbers do not behave well under restriction to subextensions, in the following sense: 

As we will prove in the next section, there is a polynomial upper bound for the condition number of cyclotomic polynomials whose conductor has a fixed number of prime factors and whose maximal coefficient (in absolute value) is polynomially bounded.

However, for $n>2$, if $\theta_n$ denotes a primitive $n$-th root, the maximal totally real subextension of the cyclotomic field $K_n$, denoted by $K_n^{+}$, is also a monogenic number field of degree $\phi(n)/2$ (\cite{washington}). Indeed, $\mathcal{O}_{K_n^{+}}=\mathbb{Z}[\psi_n]$, with $\psi_n=\theta_n+\theta_n^{-1}=cos\left(\frac{2\pi}{n}\right)>0$ if $n>4$ . Hence, for the Vandermonde matrix attached to the sequence $s=\{1,\psi_n,\psi_n^2,...,\psi_n^{\phi(n)/2-1}\}$ corresponding to the transformation between the canonical and coordinate embeddings we have, according to \cite{gautschi} as recalled above:
$$
Cond(V_s)>2^{\phi(n)/2-1}.
$$
\end{example}

\section{Asymptotic bounds for the distortion}

From now on, let us denote by $\Phi_n$ the $n$-th cyclotomic polynomial, by $K_n$ the $n$-th cyclotomic field and as usual, $\mathcal{O}_{K_n}$ its ring of integers, isomorphic to $\mathbb{Z}[x]/(\Phi_n(x))$ via evaluation in a primitive root $\zeta$. Denote also $m:=\phi(n)$, the degree of $\Phi_n$. Recall from the previous section that the transformation between the canonical and coordinate embeddings is given by the Vandermonde matrix
\begin{equation}
V_{\Phi_n}=\left(\begin{array}{cccc}
1 & \zeta_1 & \cdots & \zeta_1^{m-1}\\ 
1 & \zeta_2 & \cdots & \zeta_2^{m-1}\\  
\vdots & \vdots & \ddots \vdots\\ 
1 & \zeta_m & \cdots & \zeta_m^{m-1}
\end{array}\right),
\label{latticebij2}
\end{equation}
whose condition number measures the distortion between the embeddings.

Since, thanks to Eq. \ref{easy}, $||V_{\Phi_n}||=m$, our task is to study the quantity $||V_{\Phi_n}^{-1}||$. But before entering into the computations, we give the proof of one of the very few cases where equivalence is known in a fully satisfory manner.

\begin{thm}For $n=2^k$, the map $V_{\Phi_n}$ is a scaled isommetry. In addition, $Cond(V_{\Phi_n})=m$.
\end{thm}
\begin{proof}The proof follows the same reasoning as in Example \ref{excyclotot}, but we give the full details for the convenience of the reader, to highlight that this phenomenon is particular to this situation. So, to see that $V_{\Phi_n}$ is a scaled isometry, observe that when we multiply $V_{\Phi_n}$ by its conjugate transposed, the elements over the diagonal in the product matrix are identically $m$, and outside the diagonal, the element in position $(i,j)$ in the product matrix equals
$$
\sum_{k=0}^{m-1}\zeta_i^k\overline{\zeta_j}^k=\frac{1-\zeta_i^m\overline{\zeta_j}^m}{1-\zeta_i\overline{\zeta_j}}.
$$
But since $\zeta_i$ are $n$-primitive roots (and so are $\overline{\zeta_i}$), then $\zeta_i^m=-1$ and the sum vanishes. Hence, we have that
$$
V_{\Phi_n}V_{\Phi_n}^*=mId,
$$
and $m^{-1/2}V_{\Phi_n}$ is an isometry. For the condition number, we write $V_{\Phi_n}^{-1}=m^{-1}V_{\Phi_n}^*$, hence $||V_{\Phi_n}^{-1}||=1$. By Lemma \ref{easy}, the result follows.
\end{proof}

\subsection{General bounds} A useful expression of the inverse $V_{\Phi_n}^{-1}$ is as follows:
\begin{lem}Notations as before, $V_{\Phi_n}^{-1}=(w_{ij})$ with 
$$
w_{ij}=(-1)^{m-i}\frac{e_{m-i}(\overline{\zeta_j})}{\prod_{k\neq j}(\zeta_j-\zeta_k)},
$$
where $e_{m-i}$ is the elementary symmetric polynomial of degree $m-i$ in $m-1$ variables and $\overline{\zeta_j}=(\zeta_1,\zeta_2,...\zeta_{j-1},\zeta_{j+1},...,\zeta_m)$.
\end{lem}
\begin{proof}Starting with the definition of the product matrix and the inverse, we have $\sum_{k=1}^mw_{kj}\zeta_i^{k-1}=\delta_{ij}$, hence, the polynomial $P_j(x)=\sum_{k=1}^mw_{kj}x^{k-1}$ is the $j$-th Lagrange basis polynomial for the Vandermonde nodes and the $j$-th row of $V_{\Phi_n}^{-1}$ consists of the coefficients of $P_j$. The result holds from identifying the coefficients of the Lagrange basis polynomials
$$
P_j(x)= \prod_{k\neq j}\frac{x-\zeta_m}{\zeta_j-\zeta_m}.
$$
\end{proof}

\begin{rem} Notice that for each $1\leq i\leq m$, up to sign, the denominator of $w_{ij}$ equals $\Phi_n'(\zeta_j)$.
\label{denomcyclo}
\end{rem}

Another useful result:

\begin{lem} Let $E_i(\overline{\zeta})$ denote the symmetric function of degree $i$ in all the roots $\zeta_k$ for $1\leq k\leq m$ of an irreducible polynomial of degree $m$.Then
$$
\begin{array}{l}
E_1(\overline{\zeta})=\zeta_j+e_1(\overline{\zeta_j}),\\
E_i(\overline{\zeta})=\zeta_je_{i-1}(\overline{\zeta_j})+e_i(\overline{\zeta_j})\mbox{ for }2\leq i\leq m-1,\\
E_m(\overline{\zeta})=\zeta_je_{m-1}(\overline{\zeta_j}).
\end{array}
$$
\label{symcompleted}
\end{lem}
\begin{proof}For $i=1,m$ the result is obvious, as $E_1(\overline{\zeta})$ and $E_m(\overline{\zeta})$ are respectively the trace and norm of any of the $\zeta_k$. For $1<i<m$, $E_i(\overline{\zeta})$ is the sum of a) products of polynomial expressions in the roots of degree $i$ containing $\zeta_j$, namely $\zeta_je_{i-1}(\overline{\zeta_j})$ and b) products of expressions, also of degree $i$ not containing $\zeta_j$, namely $e_i(\overline{\zeta_j})$.
\end{proof}
We use this result to bound the numerators of the $w_{ij}$ in terms of the coefficients of the cyclotomic polynomial.
\begin{defn}For $n\geq 2$, let $A(n)$ denote from now on the maximum coefficient of $\Phi_n(x)$ in absolute value.
\end{defn}
\begin{prop}For $n\geq 2$, $m=\phi(n)$  and $1\leq k\leq m$, we have the upper bound
$$
|e_{m-k}(\overline{\zeta}_j))|\leq (k-1)A(n)+1.
$$
\label{numerator}
\end{prop}
\begin{proof}By induction, we start observing that $|E_{m}(\overline{\zeta})|=1$ (as it is the product of all primitive roots) and from Lemma \ref{symcompleted} we have $|e_{m-1}(\overline{\zeta_j})|=1$. Now assuming the result for $k$, since $|E_{m-k}(\overline{\zeta})|\leq A(n)$, we have again by Lemma \ref{symcompleted}
$$
|e_{m-(k+1)}(\overline{\zeta_j})|\leq |e_{m-k}(\overline{\zeta_j})|+|E_{m-k}(\overline{\zeta})|\leq kA(n)+1.
$$
\end{proof}
As we show next, to obtain a satisfactory bound of $w_{ij}$, we can reduce lower bounds for $|\Phi_n'(\zeta_j)|$ to lower bounds for $|\Phi_{rad(n)}'(\zeta_j^{n/rad(n)})|$.
\begin{prop}For $n\geq 2$, the following upper bound is valid:
$$
|w_{ij}|\leq rad(n)\frac{A(n)+1}{|\Phi'_{rad(n)}(\zeta_j^{n/rad(n)})|}.
$$
\label{bound1}
\end{prop}
\begin{proof}Using Proposition \ref{polycyclo2}, we express $\Phi_n(x)=\Phi_{rad(n)}(x^{n/rad(n)})$, taking derivative and evaluating at $\zeta_j$ gives
$$
|\Phi_n'(\zeta_j)|=\frac{n}{rad(n)}|\Phi_{rad(n)}'(\zeta_j^{n/rad(n)})|.
$$
Using now Proposition \ref{numerator} the result follows.
\end{proof}
Hence we are confined to lower bound $|\Phi_n'(\zeta_j)|$ for $n$ square-free, what we do next. The main ingredient is the following result:

\begin{thm}[Bateman, \cite{bateman2}] Let $n=p_1...p_k$ with $p_1<...<p_k$. Then
$$
A(n)\leq  n^{2^{k-1}}. 
$$
\label{thacta}
\end{thm}
With this remarkable result at hand we can prove this technical lemma:
\begin{lem}Let $n=p_1...p_k$ with $p_1<...<p_k$. Then
$$
\frac{1}{|\Phi_n'(\zeta_j)|}\leq n^{2^k+k}.
$$
\label{lemabrace}
\end{lem}
\begin{proof}For $k=1$, writing $\Phi_p(x)=\frac{x^p-1}{x-1}$ we have that $1/|\Phi_p'(\zeta_j)|=\frac{|\zeta_j-1|}{p}\leq \frac{2}{p}$, so the result is clear. 

Now, for $k>1$, use Lemma \ref{polycyclo2} to write:
$$
\Phi_{p_1...p_k}(x)=\frac{\Phi_{p_1...p_{k-1}}(x^{p_k})}{\Phi_{p_1...p_{k-1}}(x)}.
$$
Taking the derivative at $\zeta_j$ and observing that $\zeta_j^{p_k}$ is a $p_1...p_{k-1}$-th primitive root, we have:
$$
\frac{1}{|\Phi_n'(\zeta_j)|}\leq\frac{|\Phi_{p_1...p_{k-1}}(\zeta_j)|}{p_k|\Phi'_{p_1...p_{k-1}}(\zeta_j^{p_k})|}\leq\frac{2A(p_1...p_{k-1})\phi(p_1...p_{k-1})}{p_k|\Phi'_{p_1...p_{k-1}}(\zeta_j^{p_k})|},
$$
which by Theorem \ref{thacta} can be upper bounded as
$$
\frac{1}{|\Phi_n'(\zeta_j)|}\leq \frac{ 2n^{2^{k-1}+1}}{p_k|\Phi'_{p_1...p_{k-1}}(\zeta_j^{p_k})|}\leq \frac{ n^{2^{k-1}+1}}{|\Phi'_{p_1...p_{k-1}}(\zeta_j^{p_k})|}.
$$
Repeating the argument for $1/|\Phi'_{p_1...p_{k-1}}(\zeta_j^{p_k})|$, iterating $k-1$ times and summing up the geometric series in the exponent of $n$ yields the result.
$$
\frac{1}{|\Phi_n'(\zeta_j)|}\leq  n^{2^k+k}.
$$
\end{proof}
We can now prove the main result of this section:
\begin{thm}Denote $m=\phi(n)$ as usual. If $rad(n)=p_1...p_k$, then:
$$
Cond(V_{\Phi_n})\leq 2rad(n)n^{2^k+k+2}A(n).
$$
\label{main2}
\end{thm}
\begin{proof}We start using Prop. \ref{bound1}:
$$
|w_{ij}|\leq\frac{2rad(n)A(n)}{|\Phi_{rad(n)}'(\zeta_j^{n/rad(n)})|}.
$$
Next, by Lemma \ref{lemabrace}, we can write
$$
|w_{ij}|\leq 2rad(n)n^{2^k+k}A(n)
$$
and the result now follows.
\end{proof}
\begin{cor}Let $k\geq 1$ be fixed. If $n$ is the product of at most $k$ different primes, then $Cond(V_{\Phi_n})$ is polynomial in $n$. More in general, let $\mathcal{F}_k$ be a family of cyclotomic polynomials whose degree is divisible by at most $k$ different primes.  Assume that $A(n)=\mathcal{O}(n^r)$ for polynomials in $\mathcal{F}_k$. Then,
$$
Cond(V_{\Phi_n})=\mathcal{O}(n^{2^k+k+3+r}).
$$
\label{general}
\end{cor}
\begin{proof}
If $rad(n)$ is the product of at most a $k$, combining Theorem \ref{thacta}, Proposition 2.5 (observing that $A(n)=A(rad(n)$), and Theorem \ref{main2} we can upper bound $Cond(V_{\Phi_n})$ by $2rad(n)n^{2^k+k+2}n^r$, from which the result becomes straightforward.
\end{proof}

For instance, when $n=p^k$ the coefficients of $\Phi_n(X)$ are $0,1$. This is immediate for $k=1$ and it follows from Prop. \ref{polycyclo} for $k>1$. In this case, the general bound for $Cond(V_{\Phi_n})$ of Cor. \ref{general} is $\mathcal{O}(n^5)$.

When $n$ is the product of at most two odd prime factors, due to a 1883 result by Migotti, the coefficients of the cyclotomic polynomial belong to the set $\{0,\pm 1\}$ and likewise, from Prop. \ref{polycyclo2} the same holds for $n=p^rq^l$ with $r,l>1$. In this case, the general bound for $Cond(V_{\Phi_n})$ of Cor. \ref{general} is $\mathcal{O}(n^8)$.

Further, for $n=p^mq^lr^ts^u$ with $p<q<r<s$, we have $A(n)<n$. As before, the argument is reduced to the square free case via Prop. \ref{polycyclo2} and the square free case is due to Bang (\cite{bang}) for products of three primes and to Bloom (\cite{bloom}) for products of four primes. Hence,  in this case, the general bound in these cases is at most $\mathcal{O}(n^{23})$. 

We can refine these polynomial bounds for $k=1,2,3$ in the next section, but the moral of Theorem \ref{main2} is clear: whenever one finds that cyclotomic polynomials whose degree is divisible by at most a fixed number of primes has an explicit polynomially bounded $A(n)$, one immediately has an rough but still polynomial bound for the distortion of the corresponding RLWE/PLWE-problem. 


\section{Sharper bounds for $k\leq 3$.}
In this section we obtain more precise bounds for the condition numbers $V_{\Phi_n}$ where $rad(n)$ is product of $k$ primes with $k\leq 3$. Our arguments can be easily generalised to other cases where $A(n)$ is still polynomial. Denote as usual $m=\phi(n)$.

\subsection{The case $k=1$.} For $n=p^l$, $p>2$, let us  apply Prop.\ref{bound1} to obtain:
\begin{equation}
|w_{i,j}|\leq \frac{2p}{|\Phi'_p(\zeta_j^{p^{l-1}})|}\leq 4.
\label{aux1}
\end{equation}
In particular, for $l=1$ we have:
$$
Cond(V_{\Phi_p})\leq 4(p-1)^2.
$$
Next, we tackle the case $l>1$. 
\begin{thm}Let $n=p^l$, $m=\phi(n)$ and $l>1$. Then
$$
Cond(V_{\Phi_n})\leq 4(p-1)m.
$$
\label{oneprime}
\end{thm}
\begin{proof}
Instead of using Prop. \ref{bound1} we provide a finer bound. We start by writing
$$
|w_{i,j}|=\frac{|e_{m-i}(\zeta_j)|}{|\Phi'_{p^l}(\zeta_j)|}.
$$
To bound the numerator, observe first that for $1\leq s<p^{l-1}$ and for $1\leq r\leq p-1$, we have $E_{p^{l-1}(p-r)-s}=0$  and $|E_{p^{l-1}(p-r)}|=1$. Using Lem. \ref{symcompleted} we obtain
$$
|e_{m-i}(\zeta_j)|\leq p-1.
$$
Now, since $\Phi_{p^l}(x)=\Phi_p(x^{p^{l-1}})$, taking derivative at $\zeta_j$ gives
$$
|w_{i,j}|\leq \frac{2(p-1)}{p^l }\leq \frac{2}{p^{l-1}},
$$
thus $||V_{\Phi_n}^{-1}||\leq 2(p-1)$ and the result follows.
\end{proof}

\subsection{The case $k=2$.} For $n=pq$, we again apply Prop.\ref{bound1} together with Mogotti's theorem $A(n)\in\{0,\pm 1\}$ to obtain:
\begin{equation}
|w_{i,j}|\leq \frac{2m}{|\Phi_{pq}'(\zeta_j)|}.
\label{aux1}
\end{equation}
Since $\Phi_{pq}(x)=\frac{\Phi_p(x^q)}{\Phi_p(x)}$, one has
$$
|w_{i,j}|\leq \frac{2m}{|\Phi_{pq}'(\zeta_j)|}=\frac{2m|\Phi_p(\zeta_j)|}{q|\Phi_p'(\zeta_j^q)|}\leq \frac{4m(p-1)}{pq}\leq 4m.
$$
and $Cond(V_{\Phi_n})\leq 4m^3$.

\begin{rem}If $n=2pq$, for odd $p$ and $q$, it is easy to see that $\Phi_n(x)=\Phi_{pq}(-x)$, so a minor modification of our previous analysis still applies.
\end{rem}
And another finer bound for $n=p^rq^l$ is as follows:

\begin{thm}Let $n=p^rq^l$, with odd primes $p$ and $q$ and $m=\phi(n)$. Then
$$
Cond(V_{\Phi_n})\leq 2\phi(rad(n))m^2.
$$
\label{twoprimes}
\end{thm}
\begin{proof}We use Prop. \ref{polycyclo2}:
$$
\Phi_n(x)=\Phi_{pq}(x^{p^{r-1}q^{l-1}}),
$$
hence by using recursively Lemma \ref{symcompleted} we have upper bound:
$$
|e_i(\overline{\zeta_j})|\leq m\mbox{ for }1\leq i\leq m-1.
$$
Hence
$$
|w_{ij}|\leq\frac{m}{p^{r-1}q^{l-1}|\Phi'_{pq}(\zeta_j^{p^{r-1}q^{l-1}})|},
$$
which by Prop \ref{polycyclo}, equals $\frac{m|\Phi_p(\zeta_j^{p^{r-1}q^{l-1}})|}{p^{r-1}q^{l-1}|\Phi_p'(\zeta_j^{p^{r-1}q^{l}})|}$, hence
$$
|w_{ij}|\leq\frac{2m|\Phi_p(\zeta_j^{p^{r-1}q^{l-1}})|}{p^rq^{l-1}}\leq \frac{2m}{p^{r-1}q^{l-1}},
$$
hence $||V_{\Phi_n}^{-1}||\leq 2\phi(rad(n))m$ and the result follows.
\end{proof}

\subsection{The case $k=3$.} The starting point is the following result:

\begin{thm}[Bang, \cite{bang}] For three different odd prime numbers $p<q<r$, it holds that
$$
A(pqr)\leq p-1.
$$
\label{thmbang}
\end{thm}

In this case, we have:
\begin{prop}For $n=pqr$ it holds that $Cond(V_{\Phi_n})\leq 2m^4$.
\label{prop3f}
\end{prop}
\begin{proof}Applying recursively Lemma \ref{symcompleted}, and Theorem \ref{thmbang} we obtain:
$$|e_i(\overline{\zeta_j})|\leq (m-i)(p-1).$$
To bound the denominators we apply Proposition \ref{polycyclo2}, which is our case gives:
$$
\Phi_{pqr}(X)=\frac{\Phi_{qr}(X^p)}{\Phi_{qr}(X)}=\frac{\Phi_{qr}(X^p)\Phi_r(X)}{\Phi_r(X^q)}=\frac{\Phi_q(X^{pr})\Phi_r(X)}{\Phi_r(X^q)\Phi_q(X^p)}.
$$
Now, taking the derivative and evaluating at $\zeta_j$, we obtain:
\begin{equation}
\Phi_{pqr}'(\zeta_j)=\frac{pr\zeta_j^{pr-1}\Phi'_q(\zeta_j^{pr})\Phi_r(\zeta_j)}{\Phi_r(\zeta_j^q)\Phi_q(\zeta_j^p)}=\frac{pr\zeta_j^{pr-1}\Phi'_q(\zeta_j^{pr})}{\Phi_{qr}(\zeta_j)\Phi_q(\zeta_j^p)}.
\label{inter}
\end{equation}
Hence, $1/|\Phi_{pqr}'(\zeta_j)|\leq 2q/p$ and
\begin{equation}
|w_{ij}|\leq 2q(m-1)\leq 2(m-1)^2,
\end{equation}
Thus $Cond(V_{\Phi_n})\leq 2 m^4$.
\end{proof}
\begin{thm}For $n=p^lq^sr^t$, it holds that $Cond(V_{\Phi_n})\leq 2\phi(rad(n))^2m^2$.
\label{threeprimes}
\end{thm}
\begin{proof}As in previous case we have:
$$
|e_i(\overline{\zeta_j})|\leq mp.
$$
As for the denominator, from Prop. \ref{polycyclo2}:
\begin{equation}
|\Phi'_n(\zeta_j)|=p^{l-1}q^{s-1}r^{t-1}|\Phi'_{pqr}(\zeta_j^{p^{l-1}q^{s-1}r^{t-1}})|,
\label{final}
\end{equation}
and since $\zeta_j^{p^{k-1}q^{l-1}r^{t-1}}$ is a primitive $pqr$-root of unity, by Equation (\ref{inter}) we have:
$$|w_{ij}|\leq\frac{2mprq^2}{p^lq^sr^t}\leq\frac{2mq}{p^{l-1}q^{s-1}r^{t-1}}\leq 2\phi(rad(n))^2$$
and the result follows.
\end{proof}

\section{A subexponential asymptotic bound}
In sum and discourangingly to prove that the condition number is polynomially bounded, we currently have to impose that the number of prime factors of the conductor is fixed, to ensure that $A(n)$ is polynomial in $n$ (with $k$ kept fixed). The reason is that, in general, the coefficients of cyclotomic polynomials tend to increase very fast. The following result shows in particular that $A(n)$ is a superpolynomial function:

\begin{thm}[Maier \cite{maier}] For any $N>0$, there are $c(N)>0$ and $x_0(N)\geq 1$ such that for all $x\geq x_0(N)$, we have a lower bound
$$
|\{n\leq x: A(n)\geq n^N\}|\geq c(N)x.
$$
\end{thm}

But even worse, Erd\"os showed in \cite{erdos} that there exist infinitely many $n$ such that $A(n)\geq e^{e^{c \log(n)/\log \log(n)}}$ where $c$ can be taken to be $\log(2)$ (\cite{vaughan}).

These results highlight how difficult is to control the condition number of cyclotomic fields (not to mention the case of arbitrary Galois number fields), at least with the current available techniques. For the moment, we content ourselves with a suboptimal result to  close our study: namely, that, in general, condition numbers of cyclotomic polynomials grow subexponentially in the degree. We need the following result:

\begin{thm}[Bateman \cite{bateman}] For any $\varepsilon> 0$, there exists $n_0 \geq 2$ such that for $n\geq n_0$ we have 
$$
A(n)\leq e^{n^{(1+\varepsilon)\log(2)/\log \log(n)}}.
$$
\label{batemanth}
\end{thm}
To give our final result we need the following:
\begin{defn}[The prime-divisor function] For $n\in\mathbb{N}$, $\omega(n):=$number of different primes dividing $n$.
\end{defn}
It is well known that $\omega(n)=O(\log(n)/\log \log(n))$. An argument for this is as follows: for any $x\in\mathbb{R}$ denote by $x\sharp$ its primorial, namely, the product of all the primes below $x$. Then 
\begin{equation}
\omega(n)\leq \omega (\log(n)\sharp)=\pi(\log(n))\approx\frac{\log(n)}{\log \log(n)},
\label{omega}
\end{equation}
where $\pi(n)$ stands for the number of primes below $n$ and the asymptotic last approximation is due to the Prime Number Theorem.
We can now close with the following result:
\begin{thm}For $n\geq 2$, we have
$$
Cond(V_{\Phi_n})=\mathcal{O}\left( n^{n^{\frac{1}{\log\log(n)}}+\frac{\log(n)}{\log \log(n)}+3}e^{n^{\frac{1}{\log \log(n)}}}\right).
$$

\label{final}
\end{thm}
\begin{proof}First, from Theorem \ref{main2}, we have
$$
Cond(V_{\Phi_n})=O(rad(n)n^{2^k+k+2})A(n).
$$
Secondly, from Theorem \ref{batemanth} choosing $\varepsilon$ small enough, we can write
$$
Cond(V_{\Phi_n})=O(n^{2^k+k+3}e^{n^{\log(2)/\log \log(n)}})
$$
Finally, since $k=\omega(n)$, from Eq. \ref{omega} the result follows.
\end{proof}


\begin{thebibliography}{99}

\bibitem{bang}A.S. Bang: Om ligningen $\Phi_m(X)=0$. \emph{Nyt tidsskrift for Matematik, Afdeling B} (1895), 6--12.
\bibitem{bateman}P.T. Bateman: Note on the coefficients of cyclotomic polynomials. \emph{Bull. Amer. Math. Soc.} 55 (1949) n. 12, 1180--1181.
\bibitem{bateman2}P.T. Bateman: On the size of the coefficients of the cyclotomic polynomial.\emph{Seminaire de Th\'eorie des Nombres de Bordeaux, 11} (28) (1982) 1--18. 
\bibitem{bernstein} D.J. Bernstein, C. Chuengsatiansup, T. Lange, C. van Vredendaal: NTRU Prime (2016). http://eprint.iacr.org/2016/461
\bibitem{bloom} D.M. Bloom: On the coefficients of the cyclotomic polynomial. \emph{American Mathematical Monthly}, 75 (4), 372--377 (1968).
\bibitem{boas} P. E. Boas. Another NP-Complete Problem and the Complexity of Computing Short Vectors in a Lattice. Tech. Report 81-04, Mathematische Instituut, University of Amsterdam, 1981.
\bibitem{DD} L. Ducas, A. Durmus. Ring-LWE in polynomial rings. In PKC, 2012.
\bibitem{statistics}N.C. Dwarakanath, S. D. Galbraith. Sampling from discrete Gaussians for lattice-based cryptography on a constrained device. Preprint: https://www.math.auckland.ac.nz/~sgal018/gen-gaussians.pdf
\bibitem{erdos}P. Erd\"os: On the coefficients of the cyclotomic polynomial. \emph{Portugaliae Mathematica,} 8 (1949), n. 2, 63--71.
\bibitem{gautschi}W. Gautschi, G. Inglese: Lower bounds for the condition number of Vandermonde matrices. \emph{Numerische Mathematik}, 52 (1988), 241--250.
\bibitem{LPR} V. Lyubashevsky, C. Peikert, O. Regev. On ideal lattices and learning with errors over rings. In: Gilbert H. (eds) \emph{Advances in Cryptology – EUROCRYPT 2010.} Lecture Notes in Computer Science, 6110. Springer.
\bibitem{maier} H. Maier. Cyclotomic polynomials with large coefficients. \emph{Acta arithmetica}, 64 (3) (1993) 227--235.
\bibitem{pan}V. Y. Pan: How bad are Vandermonde matrices? \emph{SIAM journal on matrix analysis and applications}, 37 (2), (2016) 679--694.
\bibitem{PRS} C. Peikert, O. Regev, N. Stephens-Davidowitz. Pseudorandomness of Ring-LWE for any ring and modulus. In STOC, 2017.
\bibitem{regev} O. Regev. On lattices, learning with errors, random linear codes and cryptography. J. ACM, 56 (6), 2009.
\bibitem{RSW} M. Rosca, D. Stehl\'e, A. Wallet. On the ring-LWE and polynomial-LWE problems. In: Nielsen J., Rijmen V. (eds) \emph{Advances in Cryptology – EUROCRYPT 2018.} Lecture Notes in Computer Science, vol 10820. Springer.
\bibitem{stehle2}D. N. Stehle, R. Steinfeld, K. Tanaka, K. Xagawa. Efficient public key encryption based on ideal lattices. In \emph{Advances in Cryptology ASIACRYPT 2009}. 617--635 (2009).
\bibitem{stewart} I. Stewart. \emph{Algebraic number theory and Fermat's last theorem.} AK Peters Ltd, 2002.
\bibitem{vaughan}R.C. Vaughan: Bounds for the coefficients of cyclotomic polynomials. The Michigan Mathematical Journal, 21 (1975) n.4, 289--295.
\bibitem{washington} L.C. Washington. \emph{Introduction to cyclotomic fields.} Springer GTM, 1997.


\end{thebibliography}
\end{document}